\documentclass[reqno,11pt]{amsart}
\usepackage{amssymb}
\usepackage{verbatim}
\usepackage[all,cmtip]{xy}
\newif\ifpdf
\ifpdf
  \usepackage[pdftex]{graphicx}
  \usepackage[pdftex]{hyperref}
\else
  \usepackage{graphicx}
\fi

\numberwithin{equation}{section}       
 \theoremstyle{plain}    
 \newtheorem{thm}{Theorem}[section]
 \numberwithin{equation}{section} 
 \numberwithin{figure}{section} 
 \theoremstyle{plain}
 \theoremstyle{plain}    
 \theoremstyle{plain}    
 \newtheorem{pro}[thm]{Proposition} 
 \theoremstyle{plain}    
 \newtheorem{lem}[thm]{Lemma} 
 \theoremstyle{remark}
 \newtheorem{rem}[thm]{Remark}
 \theoremstyle{definition}
\newtheorem{ex}[thm]{Example}

\newtheorem*{thmA}{Theorem A} 
\newtheorem*{thmB}{Theorem B} 
 
\newtheorem*{thmD}{Theorem D}

\newtheorem*{proC}{Proposition C}

\theoremstyle{definition}
\newtheorem{defi}[thm]{Definition}

\newtheorem*{ackn}{Acknowledgement}

\newcommand{\R}{{\mathbb{R}}}

\newcommand{\cE}{{\mathcal{E}}}

\newcommand{\cZ}{{\mathcal{Z}}}

\renewcommand{\a}{\alpha}
\renewcommand{\b}{\beta}

\newcommand{\e}{\varepsilon}
\newcommand{\f}{\varphi}

\newcommand{\MA}{\mathrm{MA}\,}
\newcommand{\Amp}{\mathrm{Amp}\,}

\newcommand{\vol}{\operatorname{vol}}

\newcommand{\supp}{\operatorname{supp}}

\newcommand{\ax}{\alpha_{X}}
\newcommand{\ay}{\alpha_{Y}}

\newcommand{\pistar}{\pi^{\star}}
\newcommand{\pistarb}{\pi_{\star}}
\newcommand{\te}{\theta}
\begin{document}

\setcounter{tocdepth}{1}

\title{Stability of Monge-Amp{\`e}re energy classes}

\date{\today}

\author{E.~Di Nezza}

\address{I.M.T., Universit{\'e} Paul Sabatier\\
31062 Toulouse cedex 09\\
France}

\email{eleonora.dinezza@math.univ-toulouse.fr}

\address{Universit\`a di Roma `Tor Vergata '\\
Rome\\
Italy}
\email{dinezza@mat.uniroma2.it}

\begin{abstract}
We show that the non pluripolar product of positive currents is a bimeromorphic invariant. Under some natural assumptions, we show that the (weighted) energy associated to big cohomology classes are also bimeromorphic invariants. We compare the weighted energy functionals of currents w.r.t to different cohomology classes and establish quantitative estimates between big capacities.
\end{abstract} 

\maketitle

\tableofcontents

\newpage

\section*{Introduction}
Let $X$ be a compact $n$-dimensional K{\"a}hler manifold, $T_1=\theta_1+dd^c\f_1,...,T_p=\theta_p+dd^c\f_p$ be closed positive $(1,1)$-currents and $\theta_1+dd^c V_{\theta_1}, ..., \theta_p+dd^c V_{\theta_p} $ be canonical currents with minimal singularities. Following the construction of Bedford-Taylor \cite{bt1} in the local setting, it has been shown in \cite{begz} that 
$$ {\bf 1}_{\bigcap_j\{\f_j>V_{\theta_j}-k\}}(\theta_1+dd^c \max(\f_1, V_{\theta_1}-k))\wedge...\wedge (\theta_p+dd^c\max(\f_p, V_{\theta_p}-k))$$ is non-decreasing in $k$ and converge to the so called \emph{non-pluripolar product} $$\langle T_1\wedge...\wedge T_p\rangle.$$
The resulting positive $(p,p)$-current does not charge pluripolar sets and
it is always \emph{well-defined} and \emph{closed}.\\
\indent Given $\a$ a big cohomology class, a positive closed $(1,1)$-current $T\in\a$ is said to have \emph{full Monge-Amp{{\`e}}re mass} if
$$\int_X \langle T^n \rangle=\vol(\a)$$
and we then write $T\in \cE(X,\a)$.
In \cite{begz} the authors define also \emph{weighted energy functio\-nals} $E_{\chi}$ (for any weight $\chi$) in the general context of a big class extending the case of a K{\"a}hler class (\cite{GZ2}). The space of currents with finite weighted energy is denoted by $\cE_{\chi}(X,\a)$.

\indent The aim of the present paper is to show the invariance of the non-pluripolar product and establish stability properties of energy classes.

\begin{thmA}
\emph{The non-pluripolar product is a bimeromorphic invariant.}\\
More precisely, fix $\a\in H^{1,1}(X,\R)$ a big class and $f:X-->Y$ a bimeromorphic map, then 
\begin{itemize}
\item[1)] $f_{\star} \langle T^n\rangle=\langle (f_{\star} T)^n\rangle$ for any positive closed  $T\in\a$.
\end{itemize}
Furthermore if $f_{\star} \Big( \mathcal{T}_{\a}(X)  \Big)= \mathcal{T}_{f_{\star} \a}(Y)$ then
\begin{itemize}
\item[2)] $f_{\star}(\cE(X,\a))= \cE(Y,f_{\star}\a)$;
\item[3)]  $f_{\star}(\cE_{\chi}(X,\a))= \cE_{\chi}(Y,f_{\star}\a)$ for any weigth $\chi\in {\mathcal W}^-\cup {\mathcal W}^+_M$.
\end{itemize}
\end{thmA}
\indent Here $\mathcal{T}_{\a}(X)$ denotes the set of all positive and closed currents in the big class $\a$ and $\mathcal{T}_{f_{\star} \a}(Y)$ is the set of all positive closed currents in the image class. The Condition on the image of positive currents insures that the push-forward of a current with minimal singularities is still with minimal singularities:
this easily implies that the volumes are preserved, i.e. $\vol(\a)=\vol(f_{\star} \a)$. We show conversely in Propostion \ref{deczar} that the condition $f_{\star} \Big( \mathcal{T}_{\a}(X)  \Big)= \mathcal{T}_{f_{\star} \a}(Y)$ is equivalent to $\vol(\a)=\vol(f_{\star} \a)$ in complex dimension 2, by using the existence of Zariski decompositions.

\indent A related problem is to understand what happens to the energy classes if we change cohomology classes on a fixed compact K{\"a}hler manifold. Let $\a,\b$ be big cohomology classes. Given $T\in\mathcal{T}_{\a}(X)$ and $S\in\mathcal{T}_{\beta}(X)$ so that $T+S\in\mathcal{T}_{\a+\beta}(X)$, we wonder whether 
$$T\in\cE_{\chi}(X,\a) \quad \textit{and} \quad S\in \cE_{\chi}(X,\beta)\quad \stackrel{\displaystyle\Longrightarrow}{\Longleftarrow}\quad T+S\in \cE_{\chi}(X,\a+\beta)$$  \indent It turns out that $T+S\in\cE_{\chi}(X,\a+\b)$ implies $T\in\cE_{\chi}(X,\a)$ and $S\in \cE_{\chi}(X,\beta)$ in a very general context (Proposition \ref{pro1}) but the reverse implication is false in general (see Counterexamples \ref{nostab} and \ref{ce2}). We obtain a positive answer under restrictive conditions on the cohomology classes (see Propositions \ref{pro3} and $\ref{quasih}$).
\begin{thmB}
Let $\a,\b$ be merely big classes, $T\in \mathcal{T}_{\a}(X) $, $S\in\mathcal{T}_{\b}(X)$ and $\chi\in {\mathcal W}^-\cup {\mathcal W}^+_M$. Then 
\begin{itemize}
\item[1)] $T+S\in \cE(X, \a+\beta )$ implies $T\in\cE(X, \a)$ and $S\in\cE(X, \b)$,
\item[2)] $T+S\in \cE_{\chi}(X, \a+\beta )$ implies $T\in\cE_{\chi}(X, \a)$ and $S\in\cE_{\chi}(X, \b)$.
\end{itemize}
\indent If $\a,\b$ are K{\"a}hler, conversely
\begin{itemize}
\item[3)] $T\in\cE(X, \a)$ and $S\in\cE(X, \b)$  implies  $T+S\in \cE(X, \a+\beta )$,
\item[4)] $T\in\cE_{\chi}(X, \a)$ and $S\in\cE_{\chi}(X, \b)$  implies $T+S\in \cE_{\chi}(X, \a+\beta )$.
\end{itemize}
\end{thmB}
\begin{proC}\label{prop}
Assume that $S\in \beta$ has bounded local potentials and that the sum of currents with minimal singularities in $\a$ and in $\beta$ is still with minimal singularities. If $p> n^2-1$ then
$$ T\in \mathcal{E}^p(X,\a) \Longrightarrow T+S\in \mathcal{E}^q(X,\a+\beta),$$
where $0<q<p-n^2+1$.
\end{proC}
\noindent We stress that the condition on the sum of currents having minimal singularities is not always satisfied as noticed in Remark \ref{remmin}, but it is a necessary condition if we want the positive intersection class $\langle \a\cdot\beta \rangle$ to be multi-linear (see \cite{begz}).
\indent In our proof of Proposition \textbf{C} we establish a comparison result of capacities which is of independent interest:
\begin{thmD}
Let $\a$ be a big class and $\b$ be a semipositive class. We assume that the sum of currents with minimal singularities in $\a$ and $\b$ is still with minimal singularities. 
Then, for any Borel set $K\subset X$, there exist $C>0$ such that
$$\frac{1}{C}\,Cap_{\theta_{\a,\min}}(K)\leq  \,Cap_{\theta_{\a+\b,\min}}(K)\leq C \left(Cap_{\theta_{\a,\min}}(K)\right)^{\frac{1}{n}}$$
where $\theta_{\a,\min}:=\theta_\a+dd^cV_{\theta_\a}$.
\end{thmD}
Let us now describe the contents of the article. We first introduce some basic notions such as currents with minimal singularities and finite energy classes and we recall more or less known facts, e.g. that currents with full Monge-Amp{{\`e}}re mass have zero Lelong number on a Zariski open set (Proposition \ref{lelong}).\\
\indent In Section \ref{due}, we show that the non-pluripolar product is a bimeromorphic invariant (Theorem \ref{bimer}). Furthermore, under a natural condition on the set of positive $(1,1)$-currents, we are able to prove that weighted energy classes are preserved under bimeromorphic maps (Proposition \ref{fullmass}).\\
\indent In the third part of the paper we study the stability of the energy classes (see e.g. Theorem \ref{pro1} and Proposition \ref{pro3}) and we give some counterexamples.\\
\indent Finally, we compare the Monge-Amp{{\`e}}re capacities w.r.t different big classes (Theorem \ref{compcap}) and we use this result to give a partial positive answer to the stability property of weighted homogeneous classes $\cE^p$ (Proposition \ref{quasih}).

\begin{ackn}
I would like to thank my advisors Vincent Guedj and Stefano Trapani for several useful discussions, for all the time they commit to my research and for their support. I also thank an anonymous referee who helped me clarifying Section \ref{due}.
\end{ackn}

\section{Preliminaries}
\subsection{Big classes}
Let $X$ be a compact K{\"a}hler manifold and let $\a\in H^{1,1}(X,\R)$ be a real $(1,1)$-cohomology class.

Recall that $\a$ is said to be \emph{pseudo-effective} (\emph{psef} for short) if it can be represented by a closed positive $(1,1)$-current $T$.
Given a smooth representative $\te$ of the class $\a$, it follows from $\partial\bar{\partial}$-lemma that any positive $(1,1)$-current can be written as $T=\te+dd^c \f$ where the global potential $\f$ is a $\te$-psh function, i.e. $\te+dd^c\f\geq 0$. Here, $d$ and $d^c$ are real differential operators defined as
$$d:=\partial +\bar{\partial},\qquad d^c:=\frac{i}{2\pi}\left(\bar{\partial}-\partial \right).$$
The set of all psef classes forms a closed convex cone and its interior is by definition the set of all \emph{big} cohomology classes:
\begin{defi} 
We say that $\a$ is \emph{big} if it can be represented by a \emph{K{\"a}hler current}, i.e. there exists a positive closed $(1,1)$-current $T\in\a $ that dominates a K{\"a}hler form .
\end{defi}

\subsubsection{Analytic and minimal singularities}
 
A positive current $T=\theta+dd^c\f$ is said to have \emph{analytic singularities} if there exists $c>0$ such that (locally on $X$),
$$
\f=\frac{c}{2}\log\sum_{j=1}^{N}|f_j|^2+u,
$$
where $u$ is smooth and $f_1,...f_N$ are local holomorphic functions.

\begin{defi}
{\it If $\a$ is a big class, we define its \emph{ample locus} $\Amp(\a)$ as the set of points $x\in X$ such that there exists a strictly positive current $T\in\a$ with analytic singularities and smooth around $x$.}
\end{defi}

The ample locus $\Amp(\a)$ is a Zariski open subset by definition, and it is nonempty thanks to Demaillly's regularization result (see \cite{Bou2}).

If $T$ and $T'$ are two closed positive currents on $X$, then $T$ is said to be \emph{more singular} than $T'$ if their local potentials satisfy $\f\le\f'+O(1)$.
\begin{defi} 
\it{A positive current $T$ is said to have \emph{minimal singularities} (inside its cohomology class  $\a$) if it is less singular than any other positive current in $\a$. Its $\theta$-psh potentials $\f$ will correspondingly be said to have minimal singularities.}
\end{defi}
Such $\theta$-psh functions with minimal singularities always exist, one can consider for example
$$V_\theta:=\sup\left\{ \f\,\,\theta\text{-psh}, \f\le 0\text{ on } X \right \}.$$ 

\begin{rem}\label{minsum}
Let us stress that the sum of currents with minimal singularities does not necessarily have minimal singularities. For example, consider $\pi: X \rightarrow \mathbb{P}^2 $ the blow up at one point $p$ and set $E:=\pi^{-1}(p)$. Take  
$ \a=\pi^{\star} \{\omega_{FS}\}+\{E\} $ and $\b=2\pi^{\star}\{ \omega_{FS}\}-\{E\}$ where $\omega_{FS}$ denotes the \emph{Fubini-Study} form on $\mathbb{P}^2 $. As we will see in Remark \ref{blow} currents with minimal singularities in $\a$ are of the form $S_{\min}=\pi^{\star} T_{\min}+[E] $ where $ T_{\min}$ is a current with minimal singularities in $\{\omega_{FS}\}$ (i.e. its potential is bounded) and so they have singularities along $E$. On the other hand, currents with minimal singularities in the K{\"a}hler class $\b$ have bounded potentials, hence the sum of currents with minimal singularities in $\a$ and in $\b$ is a current with unbounded potentials. But $ \a+\b= 3\pi^{\star}\{\omega_{FS}\} $ is semipositive hence currents with minimal singularities have bounded potentials.
\end{rem}
  
\subsubsection{Images of big classes.}
It is classical that big cohomology classes are invariant under pull back and push forward (see e.g. \cite[Proposition 4.13]{Bou1}).
\begin{lem}
Let $f : X--> Y$ be a bimeromorphic map and $\ax\in H^{1,1}(X,\R)$, $\ay\in H^{1,1}(Y,\R)$ be big cohomology classes. Then $f_{\star}\ax$ and $f^{\star}\ay$ are still big classes.
\end{lem}
Note that this is not true in the case of K{\"a}hler classes.
\subsubsection{Volume of big classes.}
Fix $\a\in H_{big}^{1,1}(X,\R)$. We introduce

\begin{defi}
\it{Let $T_{\min}$ a current with minimal singularities in $\a$  and let $\Omega$ a Zariski open set on which the potentials of $T_{\min}$ are locally bounded, then 
\begin{equation}\label{v}
\vol(\a):=\int_{\Omega} T_{\min}^n>0
\end{equation}
is called the volume of $\a$.}
\end{defi}
\indent Note that the Monge-Amp{\`e}re measure of $T_{\min}$
is well defined in $\Omega$ by \cite{bt} and that the volume is independent of the choice of $T_{\min}$ and $\Omega$ (\cite[Theorem 1.16]{begz}).
 
Let $f: X\rightarrow Y$ be a modification between compact K{\"a}hler manifolds and let $\ay\in H^{1,1}(Y,\R)$ be a big class. The volume is preserved by pull-backs, $$ \vol(f^{\star}\ay)=\vol(\ay) $$(see \cite{Bou1}), on the other hand, it is in general not preserved by push-forwards:
\begin{ex}\label{esempio}
Let $\pi: X\rightarrow \mathbb{P}^2$ be the blow-up along $\mathbb{P}^2$ at point $p$.	The class $\ax:=\{\pistar \omega_{FS}\}-\varepsilon \{E\}$ is K{\"a}hler whenever $0<\varepsilon<1$ and $\pistarb \ax=\{\omega_{FS}\}$. Now, $\vol(\ax)=1-{\varepsilon}^2$ while $\vol(\pistarb \ax)=1$.
\end{ex}

\subsection{Finite energy classes}
Fix $X$ a $n$-dimensional compact K{\"a}hler manifold, $\a\in H^{1,1}(X,\R)$ be a big class and $\theta\in \a$ a smooth representative.
\subsubsection{The non-pluripolar product}
Let us stress that since the non-pluripolar product does not charge pluripolar sets,
\begin{equation*}
\vol(\a)=\int_X\langle T_{\min}^n\rangle.
\end{equation*}

\begin{defi}\label{defi:fullMA} 
{\it A closed positive $(1,1)$-current $T$ on $X$ with cohomology class $\a$ is said to have \emph{full Monge-Amp{\`e}re mass} if
$$\int_X\langle T^n\rangle=\vol(\a).$$
We denote by $\cE(X,\a)$ the set of such currents.} 
{\it If $\f$ is a $\theta$-psh function such that $T=\theta+dd^c\f$. 
The \emph{non-pluripolar Monge-Amp{\`e}re measure} of $\f$ is
$$
\MA(\f):=\langle(\theta+dd^c\f)^n\rangle=\langle T^n\rangle.
$$
We will say that $\f$ has \emph{full Monge-Amp{\`e}re mass} if $\theta+dd^c\f$ has full Monge-Amp{\`e}re mass. We denote by $\cE(X,\theta)$ the set of corresponding functions.}
\end{defi}
Currents with full Monge-Amp{\`e}re mass have mild singularities.

\begin{pro}\label{lelong}
A closed positive $(1,1)$-current $T\in \cE(X,\a)$ has zero Lelong number at every point $x\in \Amp(\a)$.
\end{pro}

\begin{proof}
This is an adaptation of \cite[Corollary 1.8]{GZ2}. Let us denote $\Omega= \Amp(\a)$. We claim that for any compact $K \subset\subset \Omega$ there exists a positive closed $(1,1)$-current $T_K\in\a$ with minimal singularities and such that it is a smooth K{\"a}hler form near $K$. Fix $\te$ a smooth form in $\a$ and $T_{\min}=\te+dd^c\f_{\min}$ a current with minimal singularities. By Demailly's regularization theorem \cite{Dem1}, in the big class $\a$ we can find a strictly positive current with analytic singularities $T_0=\te+dd^c\f_0$
that is smooth on $\Omega$. Then we define
$$\f_C:=\max(\f_0,\f_{\min}-C)$$
where $C>>1$. Clearly, 
$T_C=\te+dd^c\f_C$ is the current we were looking for.
\noindent For any point $x\in\Omega$, let $K=\overline{B(x,r)}$. Let $\chi$ be a smooth cut-off function on $X$ such that $\chi \equiv 1$ on $B(x,r)\subset K$ and $\chi \equiv 0$ on $X\setminus B(x,2r)$ where $r>0$ is small. Consider a local coordinates system in a neighbourhood of $x$ and define the $\te$-psh function $\psi_{\varepsilon}=\varepsilon \chi \log\|\cdot\|+\f_C $ for $\varepsilon$ small enough. Now, if $T=\te+dd^c \f$ has positive Lelong number at point $x$, then $\f\leq \psi_{\varepsilon}$. On the other hand $T_{\varepsilon}=\te+dd^c \psi_{\varepsilon}$ does not have full Monge-Amp{\`e}re mass since 
$$ \int_{\{\psi_{\varepsilon}\leq \f_C-k\}\cap B(x,r)} \MA(\psi_{\varepsilon}^{(k)}) $$
does not converge to $0$ as $k$ goes to $+\infty$, where $\psi_{\varepsilon}^{(k)}:=\max(\psi_{\varepsilon},\f_C-k)$ are the "canonical" approximants of $\psi_{\varepsilon}$ (\cite[p.229]{begz}). Therefore by \cite[Proposition 2.14]{begz}, it follows that $T\notin \cE(X,\a)$.
\end{proof}

We say that a positive closed $(1,1)$-current $T\in\a$ is pluripolar if it is supported by some closed pluripolar set: if $T=\te +dd^c \f$, $T$ is pluripolar implies that $\supp T\subset \{\f=-\infty\}$.
\begin{lem}\label{pluricurr}
For $j=1,...,p$, let $\a_j\in H^{1,1}(X,\R)$ be a big class and $T_j\in \a_j$. If $T_1$ is pluripolar then
$$\langle T_1\wedge...\wedge T_p\rangle=0.$$ 
\end{lem}
\begin{proof}
First note that, since the non pluripolar product does not put mass on pluripolar sets, we have
$${\bf 1}_{X\setminus A} \,\langle T_1\wedge...\wedge T_n\rangle=\langle T_1\wedge...\wedge T_n\rangle $$
with $A$ the closed pluripolar set supporting $T_{1}$. Now, let $\omega$ be a K{\"a}hler form on $X$. In view of \cite[Proposition 1.14]{begz}, upon adding a large multiple of $\omega$ to the $T_j$'s we may assume that their cohomology classes are K{\"a}hler classes. We can thus find K{\"a}hler forms $\omega_j$ such that $T_j=\omega_j+dd^c\f_j$. Let $U$ be a small open subset of $X\setminus A$ on which $\omega_j=dd^c\psi_j$, where $\psi_j\le 0$ is a smooth psh function on $U$, so that $T_j=dd^cu_j$ on $U$. By definition on the plurifine open subset 
$$O_k:= \bigcap_j \{u_j>-k\} $$
we must have $ {\bf 1}_{O_k}\langle dd^c u_1\wedge...\wedge dd^c u_p\rangle ={\bf 1}_{O_k} \bigwedge _{j}dd^c\max{ (u_j,-k)}$.
Since $u_{1}$ is a smooth potential on $U$, $u_{1}>-k$ for $k$ big enough and furthermore, since $T_{1}$ is supported by $A$, we have that $dd^c u_{1}=0$. So, clearly
$${\bf 1}_{O_k} \bigwedge _{j}dd^c\max{ (u_j,-k)}=0$$
and hence the conclusion.
\end{proof}

\subsubsection{Weighted energy classes}
By a \emph{weight function}, we mean a smooth increa\-sing function $\chi:\R^-\to\R^-$ such that $\chi(0)=0$ and $\chi(-\infty)=-\infty$. We let
$$
{\mathcal W}^-:=\left\{ \chi:\R^- \rightarrow \R^- \, | \, 
\chi \text{ convex increasing, }
\chi(0)=0,\chi(-\infty)=-\infty
 \right\}
$$
and 
$$
{\mathcal W}^+:=\left\{ \chi:\R^- \rightarrow \R^- \, | \, 
\chi \text{ concave increasing, }
\chi(0)=0,\chi(-\infty)=-\infty
 \right\}
$$
denote the sets of convex/concave weights. We say that $\chi \in {\mathcal W}^+_M$ if $\exists M>0$
$$
0 \leq |t \chi'(t)| \leq M |\chi(t)| \qquad
\text{ for all } t \in \R^-.
$$

\begin{defi} 
{\it Let $\chi\in {\mathcal W}:= {\mathcal W}^-\cup {\mathcal W}^+$. We define the $\chi$-energy of a $\theta$-psh function $\f$ as 
$$
E_{\chi,\theta}(\f):=\frac{1}{n+1}\sum_{j=0}^n\int_X(-\chi)(\f-V_{\theta})\langle T^j\wedge \theta_{\min}^{n-j}\rangle \;\in\;]-\infty,+\infty]
$$
with $T=\theta+dd^c\f$ and $\theta_{\min}=\theta+dd^cV_{\theta}$. We set
$$\cE_{\chi}(X,\theta):= \{\f\in \cE(X,\theta) \;\,  | \;\, E_{\chi,\theta}(\f)<+\infty \}.$$
We denote by $\cE_{\chi}(X,\a)$ the set of positive currents in the class $\a$ whose global potential has finite $\chi$-energy. }
\end{defi} 

\noindent When $\chi\in{\mathcal W}^-$, \cite[Proposition 2.8]{begz} insures that the $\chi$-energy is non-increasing and for an arbitrary $\theta$-psh function $\f$, 
 $$
 E_{\chi,\theta}(\f):=\sup_{\psi\ge\f} E_{\chi,\theta}(\psi)\in]-\infty,+\infty]
 $$
over all $\psi\ge\f$ with minimal singularities. On the other hand, if $\chi \in{\mathcal W}^+_M$, we loose monotonicity of the $\chi$-energy function but it has been shown in \cite [p.465]{GZ2} that
$$\f\in \cE_{\chi}(X,\a)\qquad\textit{iff} \qquad \sup_{\psi\ge\f} E_{\chi,\theta}(\psi)<+\infty $$
over all $\psi$ with minimal singularities. Recall that for all weights $\chi\in {\mathcal W}^-, \tilde{\chi} \in {\mathcal W}^+$,
we have
$$
{\mathcal E}_{\tilde{\chi}}(X,\a) \subset {\mathcal E}^1(X,\a)
\subset {\mathcal E}_{\chi}(X,\a)
\subset {\mathcal E}(X,\a).
$$
For any $p>0$, we  use the notation
$$
{\cE}^p(X,\theta):=\cE_{\chi}(X,\theta),
\text{ when } \chi(t)=-(-t)^p.
$$

\section{Bimeromorphic images of energy classes}\label{due}
From now on $X$ and $Y$ denote arbitrary $n$-dimensional compact K{\"a}hler manifolds. We recall that  a bimeromorphic map $f:X-->Y$ can be decomposed as
\begin{displaymath}
\xymatrix{ & \Gamma  \ar[dl]_{\pi_1} \ar[dr]^{\pi_2} 
 \\
 X & &  Y }
\end{displaymath}
where $\pi_1,\pi_2$ are two holomorphic and bimeromorphic maps and $\Gamma$ denotes a desingularization of the graph of $f$. For any positive closed $(1,1)$-current $T$ on $X$ we set
$$f_{\star}T:= (\pi_2)_{\star}\,\pi_1^{\star} \,T.$$ For any positive closed $(p,p)$-current $S$ is not always possible to define the push forward under a bimeromorphic map. However we define $f_{\star}\langle S\rangle$ in the usual sense in the Zariski open set $V$ where $f:U\rightarrow V$ is a biholomorfism and extending to zero in $Y\setminus V$. 

\subsection{Bimeromorphic invariance of the non-pluripolar product}
The goal of this section is to show that the non pluripolar product is a bimeromorphic invariant.
\begin{thm}\label{bimer}
Let $f: X --> Y$ be a bimeromorphic map. Let $\a_1,\cdots , \a_p\in H^{1,1}(Y,\R)$ be big classes and fix $T_j$ be a positive closed $(1,1)$-current in $\a_j$. Then 
\begin{equation}\label{pb} 
 f_{\star}\langle T_1\wedge \cdots\wedge T_p\rangle=\langle f_{\star}T_1\wedge \cdots \wedge  f_{\star}T_p\rangle.
\end{equation} 
\end{thm}
\begin{proof}
By definition of a bimeromorphic map, $f$ induces an isomorphism between Zariski open subsets  $U$ and $V$ of $X$ and $Y$, respectively. By construction the non-pluripolar product does not charge pluripolar sets, thus it is enough to check (\ref{pb}) on $V$. Since $f$ induces an isomorphism between $U$ and $V$ we have $$\left(f_{\star}\langle T_1\wedge \cdots\wedge T_p\rangle\right)|_V= f_{\star}\left(\langle T_1\wedge \cdots\wedge T_p\rangle|_U\right)=  f_{\star}\langle T_1|_U\wedge \cdots\wedge T_p|_U\rangle$$ and
$$\langle f_{\star}T_1\wedge \cdots \wedge  f_{\star}T_p\rangle|_V= \langle f_{\star}(T_1|_U)\wedge \cdots \wedge  f_{\star}(T_p|_U)\rangle.$$
Now, let $\omega$ be a K{\"a}hler form on $X$. Upon adding a multiple of $\omega$ to each $T_j$ we can assume that their cohomology classes are K{\"a}hler. Thus we can find K{\"a}hler forms $\omega_j$ such that $T_j=\omega_j+dd^c \f_j$. Fix $p\in U$ and take a small open set $B$ such that $p\in B\subset U$. In the open set $B$ we can write $\omega_j=dd^c \psi_j$ so that $T_j=dd^c u_j$ on $B$ with $u_j:=\psi_j+\f_j$. We infer that $$f_{\star} \langle \bigwedge_{j=1}^p dd^c u_j\rangle= \langle  f_{\star}(dd^c u_1)\wedge \cdots \wedge  f_{\star}(dd^c u_p)\rangle.$$ Indeed on the plurifine open subset $O_k:=\bigcap_j\{u_j>-k\}$ we have
\begin{eqnarray*}
f_{\star} \left( {\bf 1}_{O_k} \langle \bigwedge_j dd^c u_j\rangle \right)&=& f_{\star}\left( {\bf 1}_{O_k} \bigwedge_j dd^c \max(u_j,-k) \right)\\
&=&{\bf 1}_{\bigcap_j \{u_j\circ f^{-1}>-k \}}  \bigwedge_j f_{\star} (dd^c \max(u_j,-k) )
\end{eqnarray*}
where the last equality follows from the fact that for any positive $(1,1)$-current $S$ with locally bounded potential $(f_{\star}S)^n=f_{\star}(S^n)$.
\end{proof}

\subsection{ Condition (\texttt{V})}
Finite energy classes are in general not preserved by bimeromorphic maps (see Example \ref{esempio}). We introduce a natural condition to circumvent this problem.
\begin{defi}
\it{Fix $\a$ a big class on $X$. Let $\mathcal{T}_{\a}(X)$ denote the set of positive closed $(1,1)$-currents in $\a$. We say that Condition (\texttt{V}) is satisfied if 
\begin{equation*}
f_{\star} \Big( \mathcal{T}_{\a}(X)  \Big)= \mathcal{T}_{f_{\star} \a}(Y)
\end{equation*}
where $\mathcal{T}_{f_{\star} \a}(Y)$ is the set of positive currents in the image class $f_{\star} \a$.}
\end{defi}
Theorem A of the introduction is a consequence of Theorem \ref{bimer} and Proposition \ref{fullmass}.
\begin{pro}\label{fullmass}
Fix $\a\in H^{1,1}_{big}(X,\R)$. If Condition (\texttt{V}) holds,
then 
\begin{itemize}
\item[(i)]$\vol(\a)=\vol(f_{\star}\a)$,
\item[(ii)] $f_{\star}(\cE(X,\a))= \cE(Y,f_{\star}\a)$,
\item[(iii)]  $f_{\star}(\cE_{\chi}(X,\a))= \cE_{\chi}(Y,f_{\star}\a)$ for any weigth $\chi\in {\mathcal W}^-\cup {\mathcal W}^+_M$.
\end{itemize}
\end{pro}
Observe that in general $\vol(\a)\leq\vol(f_{\star}\a)$ (see Example \ref{esempio}).
\begin{proof}
Fix $T_{\min}$ a current with minimal singularities in $\a$. Observe that Condition (\texttt{V}) implies that $f_{\star}T_{\min}$ is still a current with minimal singularities, thus 
$$
\vol(\a)=\int_X \langle T_{\min}^n\rangle= \int_{Y}\langle (f_{\star} T_{\min})^n\rangle= \vol(f_{\star}\a).
$$
Fix $T\in\mathcal{T}_{\a}(X)$. Using Theorem \ref{bimer}, the change of variables formula and the fact that the pluripolar product does not put mass on analytic sets we get
$$\int_X \langle T^n\rangle =\int_Y \langle(f_{\star} T)^n\rangle $$
hence by $(i)$ it follows that
$$T\in \cE(X,\a) \Longleftrightarrow f_{\star} T \in  \cE(Y,f_{\star}\a).$$
We now want to prove $(iii)$. Let $T=\theta+dd^c\f$ and $T_k=\theta+dd^c\f^{k}$ where $\f^{k}=\max(\f, V_\theta-k)$ are the canonical approximant (note they have minimal singularities and decrease to $\f$). We recall that $f$ induces an isomorphism between Zariski opens subsets $U$ and $V$, thus by $(ii)$ and the change of variables we get that for any $j=0,\cdots,n$
\begin{eqnarray*}
&&\int_X (-\chi)(\f^{k}-V_{\theta})\langle T_{k}^j\wedge {\theta_{\min}}^{n-j}\rangle = \int_U (-\chi)(\f-V_{\theta})\langle T_{k}^j\wedge \theta_{\min}^{n-j}\rangle\\
&&= \int_V (-\chi)(\f^{k}\circ f^{-1}-V_{\theta}\circ f^{-1})\langle (f_{\star}T_k)^j\wedge(f_{\star}\theta_{\min})^{n-j} \rangle
\end{eqnarray*}
hence the conclusion.
\end{proof}
Condition (\texttt{V}) is easy to understand when $f$ is a blow up with smooth center:

\begin{rem}\label{blow}
Let $\pi : X\rightarrow Y$ be a blow up with smooth center $\cZ$, let $E=\pi^{-1}(\cZ)$ be the exceptional divisor and fix a big class $\ax$ on $X$. There exists a unique $\gamma\in \R$ such that at the level of cohomology classes $\ax=\pistar \pistarb \ax+\gamma \{E\}$. Furthermore, for any $(1,1)$-current $S\in \ax$ there exists a $(1,1)$-current $T\in \pistarb \ax$ such that $S=\pistar T+ \gamma [E]$ and $S$ is positive iff $T$ is positive and $\gamma\geq -\nu(T,\cZ)$ (consequence of Proposition 8.16 in \cite{Dem} together with Corollary 1.1.8 in \cite{b}). If Condition (\texttt{V}) holds,
then any current $S_{\min}$ with minimal singularities in $ \ax$ admits the following decomposition 
$$S_{\min}= \pistar T_{\min}+\gamma[E]$$
where $T_{\min}$ is a current with minimal singularities in $\pistarb \ax$. 
When $\gamma\geq 0$, Condition (\texttt{V}) is always satisfied. On the other side, when $\gamma<0$ this is not necessarily the case since it could happen that for some positive current $T$ in $\pistarb \ax$, $\nu(T,\cZ)<-\gamma$ (see Example \ref{esempio} where $\gamma=-\varepsilon$ and $\nu(\omega_{FS},\cZ)=0$).\\
We observe indeed that Condition (\texttt{V}) is equivalent to require that every current $T_Y\in \pistarb \ax$ is such that $\nu(T_Y,\cZ)\geq -\gamma$.\\

As the first statement of Proposition \ref{fullmass} shows, there is a link between Condition (\texttt{V}) and the invariance of the volume under push forward. For example, if $\cZ\not\subseteq X\setminus\Amp(\pistarb \ax)$ then 
$$\vol(\ax)=\vol(\pistarb \ax) \Longleftrightarrow\pistarb \Big( \mathcal{T}_{\ax}(X)  \Big)= \mathcal{T}_{\pistarb \ax}(Y) .$$ Indeed $(\Longrightarrow)$ is an easy consequence of the fact that under the assumption on the volumes we can decompose any current with minimal singularities $S_{\min}\in\ax$ as $S_{\min}=\pistar T+\gamma [E]$ whith $T\in \cE(Y,\pistarb \ax)$. Proposition \ref{lelong} implies $\nu(T,\cZ)=0$, hence $\gamma\geq 0$. Let us stress that the assumption on $\cZ$ could be removed if we knew that $\nu(T,y)=\nu(T_{\min},y)$ for any $T$ with full Monge-Amp{\`e}re mass, for any $T_{\min}$ with minimal singularities in $\pistarb \ax$ and for any $y\in Y$. It is however quite delicate to get such information at points $y$ which lie outside the ample locus.
\end{rem}

\begin{pro}\label{deczar}
Let $f: X--->Y$ a bimeromorphic map between compact K{\"a}hler manifold of complex dimension $2$. Then the following are equivalent:
\begin{itemize}
\item[(i)] $\vol(\a)=\vol(f_{\star} \a)$
\item[(ii)] $f_{\star}\Big( \mathcal{T}_{\a}(X)  \Big)= \mathcal{T}_{f_{\star} \a}(Y)$.
\end{itemize}
\end{pro}

\begin{proof}
Let us recall that $(ii)$ always implies $(i)$. Furthermore by Noether's factorization theorem it suffices to consider the case of a blow-up at one point $p$. We write $\a=\pistar \pistarb \a+\gamma \{E\}$. We recall that if $\gamma\geq 0$ there is nothing to prove, we can thus assume $\gamma<0$. Let $S$ be a current with minimal singularities representing $\alpha$ and $T$ a current with minimal singularities representing $\pistarb \alpha$. By \cite[Proposition 1.12]{begz}, $\pi^*T\in \pistar \pistarb \a$ is also with minimal singularities. Note that $\pistar T$ is cohomologous to $S- \gamma [E]$. Since $\alpha$ is big, the Siu decomposition of $S$ gives in cohomology the Zariski decomposition of $\alpha$, and similarly the Siu decomposition of $\pistar T$  gives the Zariski decomposition of $\pistar \pistarb \a$ (see e.g. \cite{Bou2}).  Furthermore, since $\pistar T$ is minimal every divisor appearing in the singular part of the Siu  decomposition of $\pistar T$ also appears in the singular part of the Siu  decomposition of $S -\gamma [E]$ with larger or equal coefficients. Then we write the Siu decomposition of $S$ and of $\pistar T$ as  
$$ S = \theta + \sum_{i=1}^N \lambda_i [D_i] + \lambda_0 [E],\quad \pistar T = \tau  + \sum_{i=1}^N \eta_i [D_i] + \eta_0 [E]  $$
with $D_i \neq E$ for all $ i $,  $\lambda_i > 0 $, $ \lambda_0, \eta_i,\eta_0 \geq 0$, where in particular $\eta_0=\nu(\pistar T, E)=\nu(T,p)$. Moreover $\{\theta\}, \{\tau\}$ are big and nef classes and $\rho_i = \lambda_i - \eta_i \geq 0$, $ \rho_0 = \lambda_0-\gamma- \eta_0 \geq 0$. It follows that 
\begin{equation}\label{cohom}
\{\theta+A\} =\{\tau\}  
\end{equation}
where $ A = \sum_{i=1}^N \rho_i [D_i] + \rho_0 [E]$ is an effective $\R$ divisor. Observe that if we show $\rho_0=0$ then $\lambda_0=\eta_0+\gamma=\nu(T,p)+\gamma \geq 0$ and so we are done. Intersecting first with $\theta$ and then with $\tau$ the relation (\ref{cohom}), using the assumption on the volumes, i.e. $\{\theta\}^2 = \{\tau\}^2$, the fact that $A$ is effective, and  that $\tau$ and $\theta$ are nef, we find  
$\{\tau\} \cdot \{A\}= \{\theta\}\cdot \{A\} = 0.  $   
If we develop the square of the left hand side of (\ref{cohom}) we conclude $\{A\}^2 = 0.$ Since $\{\theta\}^2> 0$, the Hodge index theorem shows that $\{ A \} = 0$ and since $A$ is effective, it is the zero divisor. Hence $\rho_0 = 0$. 
\end{proof}
We expect that $\nu(T,x)=\nu(T_{\min},x)$ for all $x\in X$ whenever $T\in \cE(X,\a)$. We show the following partial result in this direction:
\begin{pro}
Let $X$ be a compact K{\"a}hler surface, $\a$ be a big class on $X$ and $T\in \cE(X,\a)$. Then the set $\{x\,|\; \nu(T,x)>\nu(T_{\min},x)\}$ is at most countable.
\end{pro}
\begin{proof}
We write the Siu decomposition of the current $T$ as $T=R+\sum_{j=1}^N\lambda_i [D_i]$. Note that the set $E_+(T):=\{x\in X \;|\; \nu(T,x)>0\}$ contains at most finitely many divisors (Proposition \ref{lelong}). We claim that $\{R\}$ is big and nef. Indeed, by construction the current $R$ has not positive Lelong number along curves and so any current with minimal singularities $R_{\min}\in \{R\}$ has the same property. Thus the Zariski decomposition of $\{R\}$ is of the type $\{R\}=\{R\}+0$. Furthermore $$\vol(\{R\})\leq \vol(\a)=\int_X \langle T^2\rangle =\int_X \langle R^2\rangle\leq \vol(\{R\}), $$ that implies $\vol(\a)=\{R\}^2>0$. Then 
$T= R+\sum_{j=1}^N\rho_i [D_i]+\sum_{j=1}^N \eta_i [D_i],$
where $\eta_i=\nu(T_{\min},D_i)$ with $T_{\min}\in \a$. Clearly $\rho_i\geq 0$, for any $i$. We want to show that $\rho_i=0$. Set $S:= R+\sum_{j=1}^N\rho_i [D_i]$ and write the Zariski decomposition of $\a$ as $\a=\a_1+\sum_{j=1}^N \eta_i \{D_i\}$. Then $\a_1=\{S\}$. This means that $\{S\}$ is big and nef and $\vol(\a)=\a_1^2=\{S\}^2$. Now, $\{R+A\}= \{S\}$ where $A=\sum_{j=1}^N\rho_i [D_i]$ is an effective $\R$ divisor. Using the same arguments in the proof of Proposition \ref{deczar} we get $\{A\}\cdot \{R\}=\{A\}\cdot\{S\}=\{A\}^2=0$ and using the Hodge index theorem we conclude.
\end{proof}

\section{Sums of finite energy currents }
Let $X$ be a compact K{\"a}hler manifold of complex dimension $n$ and let $\a$ and $\beta$ be big classes on $X$. Given two positive currents $T\in\a$ and $S\in\b$ with full Monge-Amp{\`e}re mass, it is natural to wonder whether $T+S$ has full Monge-Amp{\`e}re mass in $\a+\b$, and conversely.

\subsection{Stability of energy classes}

We start proving Theorem B of the introduction.

\begin{thm}\label{pro1}
Fix $T\in \mathcal{T}_{\a}(X) $, $S\in \mathcal{T}_{\b}(X)$ and $\chi\in \mathcal{W^-}\cup \mathcal{W}_M^+$. Then 
\begin{itemize}
\item[(i)] $T+S\in \cE(X, \a+\beta )$ implies $T\in\cE(X, \a)$ and $S\in\cE(X, \b)$,
\item[(ii)] $T+S\in \cE_{\chi}(X, \a+\beta )$ implies $T\in\cE_{\chi}(X, \a)$ and $S\in\cE_{\chi}(X, \b)$.
\end{itemize}
If $\a,\b$ are K{\"a}hler classes, then conversely
\begin{itemize}
\item[(iii)] $ T\in\cE(X, \a)$ and $S\in\cE(X, \b)$ implies $T+S\in \cE(X, \a+\beta ) $,
\item[(iv)]   $T\in\cE_{\chi}(X, \a)$ and $S\in\cE_{\chi}(X, \b)$ implies $T+S\in \cE_{\chi}(X, \a+\beta )$.
\end{itemize}
\end{thm}
\begin{proof}
Pick $\theta_{\a}$ and $\theta_{\b}$ smooth representatives in $\a$ and $\beta$, so that $\tilde{\theta}:=\theta_{\a}+\theta_{\b}$ is a smooth form representing $\a+\beta$. We decompose $T=\theta_{\a}+dd^c\f$ and $S= \theta_{\b} + dd^c \psi$. We assume $\f+\psi\in\cE(X, \tilde{\theta}) $, and first prove that $\f$ has full mass, which is equivalent to showing
$$m _k :=\int_{\{\f\leq \f_{\min}- k\}} \langle (\theta_{\a} + dd^c \max(\f, \f_{ \min} - k))^n \rangle\longrightarrow 0 \qquad \textit{as}\quad  k\rightarrow +\infty $$
where $T_{\min} = \theta_{\a} + dd^c \f_{\min} $  has minimal singularities in $ \a$ (\cite[p.229]{begz}). First, observe that on $X\setminus \{\psi=-\infty\}$ we have
$$\{\f\leq \f_{\min}-k\}\subseteq \{\f+\psi\leq \f_{\min}+\psi-k\}\subseteq \{\f+\psi\leq \phi_{\min}-k\}$$
where $S_{\min}= \tilde{\theta} +dd^c \phi_{\min}$ has minimal singularities in $\a+\beta$. Since the non-pluripolar product does not charge pluripolar sets, we infer
\begin{eqnarray*}
0\leq m_k & \leq & \int_{\{ \f+\psi \leq \phi_{\min}-k\}}  \langle (\theta_{\a} + dd^c \max(\f, \f_{\min}- k))^n \rangle\\
&\leq & \int_{\{ \f+\psi \leq \phi_{\min}-k\}\setminus  \{ \psi=-\infty \}} \langle (\tilde{\theta}+dd^c \max(\f+\psi, \f_{\min}+\psi - k))^n \rangle\\
&\leq & \int_{ \{\f+\psi \leq \phi_{\min}-k \} } \langle(\tilde{\theta} + dd^c \max(\f+ \psi, \phi_{\min}-k))^n \rangle
\end{eqnarray*}
 where the last inequality follows from the fact that $ \phi_{\min}$ is less singular then $\f_{\min} + \psi$ (see \cite[Proposition 2.14]{begz}). But, by assumption, the last term goes to $0$ as $k$ tends to $+\infty$, hence the conclusion. Changing the role of $\f$ and $\psi$ one can prove similarly that also $\psi$ is with full Monge-Amp{\`e}re mass.\\
 \indent We now prove the second statement. By assumption $\f + \psi \in\cE_{\chi}(X, \tilde{\theta})$ with $\chi$ a convex weight and so from above we know that $ \f $ and $\psi$ both have full Monge-Amp{\`e}re mass. It suffices to check that $\f\in \cE_{\chi}(X,\theta_{\a})$. By \cite{begz},
$$E_{\chi,\theta}(\f) < +\infty \quad \textit{iff}\quad \sup_k\int_{X}(-\chi)(\f_k -\f_{\min})MA(\f_k) < +\infty,$$
for any sequence $\f_k$ of $\theta_\a$-psh functions with full Monge-Amp{\`e}re mass decreasing to $\f$.
Since $T_1\leq T_2$ implies $\langle T_1^n\rangle \leq \langle T_2^n\rangle$ we obtain \\
$ \displaystyle
\int_{X}(-\chi)(\f_k- \f_{\min}) \langle(\theta_{\a} +dd^c \f_k)^n \rangle $
\vspace*{-9pt}
\begin{eqnarray*}
& \leq & \int_{X \setminus \{\psi=-\infty\}}(-\chi)(\f_k-\f_{\min}) \langle(\tilde{\theta} + dd^c (\f_k +\psi))^n \rangle \\
& \leq &  \int_{X\setminus \{\psi=-\infty\}} (-\chi)(\f_k +\psi -\phi_{\min} )\MA(\f_k+ \psi) 
\end{eqnarray*}
where the last inequality follows from monotonicity of $\chi$ and the fact that on $X \setminus \{\psi=-\infty\}$ 
$$\f_k -\f_{\min} = (\f_k + \psi ) -(\f_{\min} + \psi )\geq (\f_k + \psi )-\phi_{\min}.$$ 
Therefore $E_{\chi,\tilde{\theta}}(\f+\psi)<+\infty$ implies $E_{\chi,\theta_{\a}}(\f)<+\infty$, as desired.\\
\indent  Assume now that $\a,\b$ are both K{\"a}hler classes and choose K{\"a}hler forms $\omega_\a\in\a$, $\omega_\b\in\b$ as smooth representatives. We want to prove that if $\f\in\cE(X,\omega_\a)$ and $\psi\in\cE(X,\omega_\beta)$ then $\f+\psi\in\cE(X,\omega_\a+\omega_\b)$. Let $\omega$ be another K{\"a}hler form on $X$. We first show that $\f\in\cE(X,\omega_\a)$ (resp. $\f\in\cE_{\chi}(X,\omega_\a)$) if and only if $\f\in\cE(X,\omega)$ (resp. $\f\in\cE_{\chi}(X,\omega)$) whenever $\f\in PSH(X,\omega)$. We recall that, since $\omega_\a$ and $\omega$ are K{\"a}hler forms, there exists a constant $C>0$ such that $\frac{1}{C}\omega\leq \omega_\a\leq C\omega$. Thus
\begin{eqnarray*}
\int_{\{\f\leq -k\}} (\omega_\a+dd^c \f_k)^n &\leq &  \int_{\{\f\leq -k \}} (C\omega+dd^c \f_k)^n\\
 &\leq &  \tilde{C}\sum_{j=0}^n \int_{\{\f \leq -k \}} \omega^j\wedge (\omega+dd^c \f_k)^{n-j},
\end{eqnarray*}
where $\f_k:=\max(\f,-k)$. And so $\f\in\cE(X,\omega)$ implies $\f\in\cE(X,\omega_\a)$. Analogously one can prove the reverse. Similarly, for any weight $\chi\in  \mathcal{W^-}\cup \mathcal{W}_M^+ $,
\begin{eqnarray*}
\int_{X} -\chi(\f_k)(\omega_\a+dd^c \f_k)^n &\leq &  \tilde{C}\sum_{j=0}^n\int_{X} -\chi(\f_k) (\omega+dd^c \f_k)^j\wedge \omega^{n-j}.
\end{eqnarray*}
Thus, if $\f\in\cE_{\chi}(X,\omega)$ then $\f\in\cE_{\chi}(X,\omega_\a)$. With the same argument we get the reverse. Now, let $\omega$ be a K{\"a}hler form such that $\omega_\a, \omega_\b\leq \omega$. From above we have that $\f,\psi\in\cE(X,\omega)$ (resp. $\f,\psi\in\cE_{\chi}(X,\omega)$) and since the energy classes are convex (\cite[Propositions 1.6, 2.10 and 3.8]{GZ2}), it follows $\f+\psi\in \cE(X,2\omega)$ (resp. $\f+\psi\in\cE_{\chi}(X,2\omega)$). From the previous observation we can deduce $\f+\psi\in \cE(X,\omega_\a+\omega_\b)$.
\end{proof}

Examples \ref{nostab} and \ref{ce2} below show the reverse implication is not true in general. This is particularly striking if the following condition is not satisfied:

\begin{defi}
\it{We say that pseudoeffective classes $\a_1,\cdots,\a_p$ satisfy Condition $\mathcal{MS}$ if the sum $ T_{1}+\cdots +T_p$ of positive currents $T_i\in\a_i$ with minimal singularities has minimal singularities in $\a_1+\cdots +\a_p$.}
\end{defi}

Note that if $\a_1,\cdots,\a_p$ satisfy Condition $\mathcal{MS}$ the positive intersection class $\langle \a_1\cdots\a_p \rangle$ turns to be multi-linear while it is not so in general (\cite[p.219]{begz}).

\begin{pro}\label{pro3}
Let $T\in \mathcal{T}_{\a}(X) $ and $\chi\in\mathcal{W}^-\cup \mathcal{W}^-_M$. Assume that $\a $ is a K{\"a}hler class and $\b$ is a semi-positive class. Fix $\theta_\b\in\b$ a semipositive form. Then 
\begin{itemize}
\item[(i)] $T+\theta_\b\in \cE(X, \a+\beta )$ if and only if $T\in\cE(X, \a)$,
\item[(ii)] $T+\theta_\b\in \cE_{\chi}(X, \a+\beta )$ if and only if $T\in\cE_{\chi}(X, \a)$.
\end{itemize}
\end{pro}
We will exhibit an Example \ref{nostab} such that $\a$ is semipositive, $\b$ is K{\"a}hler, $\theta_\b$ is a  K{\"a}hler form in $\b$, $T\in\cE^{1}(X, \a)$ but $T+\theta_\b\notin \cE^{1}(X, \a+\beta )$.
\begin{proof}
We will first prove the second statement. Fix $\omega, \theta_{\b}$ smooth representatives of $\a$ and $\beta$, respectively and denote $\tilde{\omega}:=\omega+\theta_{\b}$. Note that $\omega$ can be chosen to be K{\"a}hler. Let $T:=\omega+dd^c\f\in \cE_{\chi}(X, \a ) $, by \cite{begz} we have
$$E_{\chi,\omega}(\f)\Longleftrightarrow \sup_{k} E_{\chi,\omega}(\f_k)<+\infty$$ where $\f_{k}:= \max(\f, 
-k)$. We now show that $E_{\chi,\tilde{\omega}}(\f_k)$ is uniformly bounded from above. Fix $A$ such that $ \tilde{\omega} \leq (A+1)\omega$. Then
\begin{eqnarray*}
&&\int_{X} -\chi(\f_k) \left(\tilde{\omega} +dd^c\f_k \right)^j \wedge \tilde{\omega} ^{n-j} \\
&& \leq  (A+1)^{n-j} \int_{X} -\chi(\f_k) \left( A\omega+ \omega +dd^c\f_k \right)^j \wedge \omega ^{n-j} \\
&& \leq  C\, \sum_{l=0}^j \int_X -\chi(\f_k) \left( \omega+dd^c\f_k\right)^{j-l} \wedge \omega^{n-j+l} \leq C' \, E_{\chi,\omega}(\f_k).
\end{eqnarray*} 

\indent The first statement is an easy consequence of the second one recalling that $$\cE(X,\a)=\bigcup_{\chi\in\mathcal{W^-}}\cE_{\chi}(X,\a).$$ The reverse inclusions is Theorem \ref{pro1}.
\end{proof}
\begin{rem}
Let us stress that the first statement of Proposition \ref{pro3} could be proved in great generality ($\a$, $\b$ big classes such that Condition $\mathcal{MS}$ holds, $\theta_\b$ current with minimal singularities) if given $\a_1,\cdots,\a_n$ big classes and $T_1\in \cE(X,\a_1)$, the following would hold
$$\int_X\langle T_1\wedge \theta_{2,\min}\wedge...\wedge \theta_{n,\min}\rangle= \int_X\langle \theta_{1,\min}\wedge...\wedge \theta_{n,\min}\rangle $$
where $\theta_{i,\min}:=\theta_i+ dd^cV_{\theta_i}\in\a_i$.
\end{rem}

\subsection{Counterexamples}

The following example shows that given two currents $T\in\cE^1 (X, \a )$ and $S\in\cE^1 (X, \beta )$ we can not expect that $T+S\in\cE^1 (X, \a+\beta )$, even if $\a$ is semipositive and $\b$ is K{\"a}hler.

\begin{ex}\label{nostab}
Let $\pi: X \rightarrow \mathbb{P}^2 $ be the blow up at one point $p$ and set $E:=\pi^{-1}(p)$. Fix $\a=\pistar \{\omega_{FS}\}$ and $\beta=2\pistar \{\omega_{FS}\}-\{E\}$ so that $\a+\beta=3\pistar \{\omega_{FS}\}-\{E\}$. We pick $\tilde{\omega}\in\a+\beta$ a K{\"a}hler form of the type $\tilde{\omega}=\pistar\omega_{FS}+\omega$, where $\omega\in\beta$ is a K{\"a}hler form. We will show that 
$$\cE^1 (X, \a ) \nsubseteq \cE^1 (X, \a+\beta ) \cap \mathcal{T}_{\a}(X).$$
The goal is to find a $\omega_{FS}$-psh function $\f$ on $\mathbb{P}^2$ such that $\pistar\f \in \cE^1 (X, \pistar\omega_{FS} )$ but $\pistar \f\notin \cE^1 (X, \tilde{\omega} )$. Let $U$ be a local chart of $\mathbb{P}^2$ such that $p \rightarrow (0, 0)\in U$. We define 
$$\f_{\delta} := \frac{1}{C}\chi \cdot u_{\delta} -K_{\delta}$$
where $u_{\delta} := -(- \log\|z\|)^{\delta}$, $\chi$ is a smooth cut-off function such that $\chi \equiv 1$ on $\mathbb{B}$ and $\chi \equiv 0$ on $U \setminus \mathbb{B}(2)$, $K_{\delta}$ is a positive constant such that $\f_{\delta}\leq −1$ and $C>0$. Choosing $C$ big enough $\f_{\delta}$ induces a $\omega_{FS}$-psh function on $\mathbb{P}^2$, say $\tilde{\f}_{\delta}$. Note that by \cite[Corollary 2.6]{cgz}
$\tilde{\f}_{\delta} \in \cE(\mathbb{P}^2, \omega_{FS})$ if $0 \leq {\delta} < 1$. We let the reader check that $\tilde{\f}_{\delta}\in W^{1,2}(\mathbb{P}^2,\omega_{FS})$ for all $0\leq \delta <1$. Therefore $\tilde{\f}_{\delta}\in \cE^1 (\mathbb{P}^2, \omega_{FS} )$ iff 
$$\int_{\mathbb{P}^2} -\tilde{\f}_{\delta} (dd^c\tilde{\f}_{\delta})^2 < +\infty$$ 
We claim this is the case iff $ 0 \leq {\delta} < \frac{2}{3}$.\\
Note that $\tilde{\f}_{\delta}$ is smooth outside $p$, therefore we have to check that 
\begin{equation}\label{int}
 \int_{\mathbb{B}(\frac{1}{2})}-u_{\delta} (dd^c u_{\delta} )^2 < +\infty.
\end{equation}
Set $\chi(t)=-(-t)^{\delta}$ so that $u _{\delta} = \chi(\log\|z\|)$. Then $(dd^c u_{\delta})^2 = C_1\, \frac{1}{8\|z\|^4}\, \chi^{''}\cdot \chi^{'}(\log\|z\|) dz_1 \wedge d\bar{z}_1 \wedge dz_2 \wedge d\bar{z}_2 $ on $\mathbb{B}(\frac{1}{2})\setminus\{(0, 0)\}$, hence the convergence of the integral in \eqref{int} is equivalent to the convergence of
$$\int_{\mathbb{B}(\frac{1}{2})\setminus\{(0, 0)\}} \frac{-\chi(\log\|z\|)\cdot \chi^{''}(\log\|z\|)\cdot \chi^{'}(\log\|z\|)}{\|z\|^4}\,dz_1 \wedge d\bar{z}_1 \wedge dz_2 \wedge d\bar{z}_2$$
$$ = \int_{0}^{\frac{1}{2}} \frac{-\chi(\log\rho)\cdot \chi^{''}(\log\rho)\cdot \chi^{'}(\log\rho)}{\rho}\,\,d\rho =\delta(1-\delta)\int_{-\log\frac{1}{2}}^{+\infty} \frac{1}{(s)^{3-3\delta}} \,ds $$
which is finite iff $0 \leq {\delta} < \frac{2}{3}$, as claimed. Therefore by Proposition \ref{fullmass} we get  $\pistar \tilde{\f}_{\delta} \in \cE^1 (X,\pistar\omega_{FS})$. But $\pistar \tilde{\f}_{\delta} \notin \cE^1 (X, \tilde{\omega})$ if $\frac{1}{2}\leq \delta < \frac{2}{3}$ since
$$\left| \nabla(\pistar\tilde{\f}_{\delta})\right| \notin L^2(X, (\tilde{\omega})^2 ) \quad \textit{if}\quad \delta \geq \frac{1}{2}.$$
Indeed, let $z =(z_1, z_2)\in \mathbb{B}$ and fix a coordinate chart in $X$, then $\pi(s,t)=(z_1, z_2)=(s,st)$. Therefore, on $\pi^{-1}(\mathbb{B})$
$$
 \f_{\delta} \circ \pi(s,t)= \frac{1}{C}\, u_{\delta}(s,st)= -\frac{1}{C}\left(-\log|s|-\log \sqrt{1 + |t|^2} \right)^{\delta}
$$
Hence,
$$\int_{\pi^{-1}(\mathbb{B})} \left| \frac{\partial(\f_{\delta}\circ \pi)}{\partial s}   \right|^2 ds\wedge d\bar{s}\wedge dt\wedge d\bar{t}\geq \left( \frac{\delta}{2C} \right)^2 \int_{\pi^{-1}(\mathbb{B})} \frac{ ds\wedge d\bar{s}\wedge dt\wedge d\bar{t} }{|s|^2 (-\log |s|)^{2-2\delta}} $$
which is not finite if ${\delta}\geq \frac{1}{2}$. The conclusion follows from \cite[Theorem 3.2]{GZ2}.
\end{ex}

\begin{rem}\label{fullmas}
Observe that $\a,\b$ satisfy Condition $\mathcal{MS}$ in previous example and also that $\pistar \tilde{\f}_{\delta} \in \cE(X, \tilde{\omega} )$. Indeed, let $T:=\pistar\omega_{FS}+dd^c(\tilde{\f}_{\delta}\circ \pi)$, we need to check that $T +\omega\in\cE(X,\a+\beta)$. Since $T\in\cE(X, \a)$ and
$$
\langle(T +\omega)^2 \rangle=\langle T^2 \rangle + 2\langle T\rangle \wedge \omega + (\omega )^2.
$$
it suffices to show that 
$$
\{\langle T \rangle \wedge \omega \} = \{\pistar \omega_{FS} \}\cdot \{\omega \}.
$$
which is equivalent to
$$\{(T -\langle T \rangle)\wedge \omega  \} = 0.$$
Hence, what we need to show is that $T -\langle T \rangle = 0$. The $(1,1)$-current $T -\langle T \rangle$ is positive and is supported by the exceptional divisor $E$. Therefore using \cite[Corollary 2.14]{Dem} it results that $$T= \langle T \rangle + \gamma [E]$$ where $\gamma=\nu(T, E)=\nu(\pistarb T, p)= 0$ since $\delta<1$. And so the conclusion.
\end{rem}
Previous remark could let us think whenever $T\in\mathcal{E}(X,\a)$ and $S\in\cE(X, \beta)$ then $T+S\in\cE(X,\a+ \beta)$, but this is not true either as the following example shows:

\begin{ex}\label{ce2}
Let $\pi: X \rightarrow \mathbb{P}^2 $ be the blow up at one point $p$ and set $E:=\pi^{-1}(p)$. Consider $\a=\pistar \{\omega_{FS}\}+\{E\}$ and $\beta=2\pistar \{\omega_{FS}\}-\{E\}$. Thus $\a+\beta=3\pistar \{\omega_{FS}\}$. Since $\beta$ is a K{\"a}hler class we can choose $S=\omega$ with $\omega$ a K{\"a}hler form. \\
Observe that currents with minimal singularities in $\a$ are of the type $\pistar S_{\min}+[E]$, where $S_{\min}$ is a current with minimal singularities in $\{\omega_{FS}\}$ (Remark \ref{blow}). By Lemma \ref{pluricurr}
$$\vol(\a)=\int_X \langle (\pistar S_{\min}+[E])^2 \rangle= \int_X\langle (\pistar S_{\min})^2 \rangle= \int_X \pistar \langle S_{\min}^2 \rangle =1,$$
while $\vol(\a+\beta)=(\a+\beta)^2=9$.

Let now $T\in\cE(X,\a)$ and recall that any positive $(1,1)$-current in $\a$ is of the form $T=\pistar S+[E]$ with $S\in\mathcal{T}_{\{\omega_{FS}\}}(\mathbb{P}^2)$. In particular we choose $T:=\pistar \omega_{FS}+[E]$. We want to show that $T+\omega\notin \cE(X,\a+\beta)$. Now, from the multilinearity of the non-pluripolar product we get
$$
\int_X \langle (T+\omega)^2\rangle = \int_X \langle (\pistar \omega_{FS}+[E]+\omega)^2\rangle = \int_X \langle (\pistar \omega_{FS}+\omega)^2\rangle=8 
$$
Hence $\int_X \langle (T+\omega)^2\rangle=8<9=\vol(\a+\beta)$.\\
\indent The same type of computations show that if we pick $T\in \cE(X,\a)$, then, for any $0<\varepsilon\leq 1 $, $T+\varepsilon \omega \notin  \cE(X, \a+\varepsilon\omega)$.
\end{ex}

\begin{rem}\label{remmin}
Note that in the latter example $\a,\b$ do not satisfy Condition $\mathcal{MS}$.
\end{rem}

\section{Comparison of Capacities}\label{capac}
Let $X$ be a compact K{\"a}hler manifold of complex dimension $n$ and let $\a$ be a big class on $X$. Set $\theta\in\a$ a smooth form and $\theta_{\min}:=\theta+dd^c V_{\theta}$ the positive $(1,1)$-current in $\alpha$ with 'canonical' minimal singularities.

\subsection{Intrinsic Capacities}\label{capac10}
We introduce the space of "$\theta_{\min}$-plurisubharmonic" functions
$$PSH(X, \theta_{\min}):=\left\{ \psi \;|\; \psi+V_{\theta} \quad\textit{is a}\;\, \theta-\textit{psh function} \right\}.$$ 
Note that a $\theta_{\min}$-psh function $\psi$ is not upper-semi-continuous but $\psi+V_{\theta}$ is.

\subsubsection{Monge-Amp{\`e}re capacity}
Following \cite{begz} we introduce the Monge-Amp{\`e}re capacity with respect to a big class.
\begin{defi}\label{cap10}
 \it{We define the capacity of a borel set $K\subseteq X$ as
\begin{equation*}
 Cap_{\theta_{\min}}(K):=\sup\left\{ \int_{K}\langle (\theta_{\min}+dd^c\psi)^n \rangle, \,\psi\in PSH(X,\theta_{\min}) \,|\, -1\leq \psi \leq 0 \right\}.
\end{equation*} }
\end{defi}
Observe that the above one is the same definition as \cite[Definition 4.3]{begz}, just taking $\psi=\f-V_{\theta}$, where $\f$ is a $\theta$-psh function. Here we introduce this equivalent formulation since in Section \ref{capac} we need the positivity of the reference current $\theta_{\min}$.

\subsubsection{The relative extremal function}
We introduce the notion of the relative extremal function with respect to $ \theta_{\min}$. If $E$ is a Borel subset of
$X$, we set
$$
h_{E, \theta_{\min}}(x):=\sup \left\{\psi(x) \, |\, \psi \in PSH(X, \theta_{\min}) , \, 
\psi \leq 0 \text{ and } \psi_{|E} \leq -1 \right\},
$$
and $$h_{E,\theta_{\min}}^*:=(h_{E, \theta_{\min}}+V_{\theta})^*-V_{\theta}.$$ It is a standard matter to show that, as in the K{\"a}hler case (see \cite{gz1}), the $\theta_{\min}$-psh function $h_{E,\theta_{\min}}^*$ satisfies
 $$Cap_{\theta_{\min}}(K)=\int_K \MA(V_{\theta}+h_{K,\theta_{\min}}^*)= \int_X (-h_{K,\theta_{\min}}^*) \MA(V_{\theta}+h_{K,\theta_{\min}}^*)$$
 where $K\subset X$ is a compact set (for details see \cite[Lemma 1.5]{bbgz}).


\subsubsection{Capacities of sublevel sets}
We now generalize \cite[Lemma 5.1]{GZ2}.
\begin{lem}\label{cap}
Fix $\chi \in {\mathcal W}^- \cup {\mathcal W}_M^+$, $M \geq 1$.
If $\f \in {\mathcal E}_{\chi}(X,\theta)$,
then 
$$
\exists C_{\f}>0, \forall t>1, \; 
Cap_{\theta_{\min}}(\f < V_{\theta} -t) \leq C_{\f} |t \, \chi(-t)|^{-1}.
$$

Conversely if there exists $C_{\f}, \e>0$ such that 
for all $t>1$,
$$ Cap_{\theta_{\min}}(\f <V_{\theta} -t) \leq C_{\f} | t^{n+\e} \, \chi(-t)|^{-1}, $$ 
then $\f \in {\mathcal E}_{\chi}(X,\theta)$.
\end{lem}

\begin{proof}
Fix $\f \in \mathcal{E}_{\chi}(X,\theta)$ and $u \in PSH(X,\theta)$
such that $-1 \leq u-V_{\theta} \leq 0$.
For $t \geq 1$, observe that by \cite[Proposition 2.14]{begz}, $\frac{\f}{t}+\left(1-\frac{1}{t} \right)V_{\theta} \in {\mathcal E}(X,\theta)$ and
$$
(\f - V_{\theta}<-2t)  \subseteq \left(\frac{\f-V_{\theta}}{t} < -1+u-V_{\theta} \right)\subseteq  (\f- V_{\theta} < -t ) .
$$
\noindent It therefore follows from the generalized comparison principle and from the multilinearity of the non-pluripolar product (\cite[Propositions 2.2 and 1.4]{begz}) that 
\begin{eqnarray*}
&&\int_{(\f - V_{\theta} <-2t)} MA(u) \leq \int_{(\f- V_{\theta} < -t )} MA \left( \frac{\f}{t}+\left(1-\frac{1}{t}\right) V_{\theta} \right) \\
&& \leq  \left(1-\frac{1}{t} \right)^n \int_{(\f- V_{\theta} < -t )} \langle \theta_{\min}^n \rangle + t^{-1} \sum_{k=1}^{n} \binom{n}{k} \int_{(\f- V_{\theta} < -t )} \langle T^k\wedge \theta_{\min}^{n-k} \rangle
\end{eqnarray*}
where $T:=\theta+dd^c\f$.
Furthermore, since 
$$ MA(V_\theta)={\bf 1}_{\{V_{\theta}=0\}}\theta^n$$ (see \cite[Corollary 2.5]{bd}), we get
$$\int_{(\f- V_{\theta} < -t )} \langle \theta_{\min}^n \rangle = \int_{(\f- V_{\theta} < -t )\cap D} \theta^n = {\bf 1}_D\theta^n (\f<-t) \leq C \omega^n(\f<-t),$$
where $D:=\{V_{\theta}=0\}$, $\omega$ is a K{\"a}hler form on $X$ and $C>0$. We recall that $\vol_{\omega}(\f<-t)$ decreases exponentially fast (see \cite{gz1}) and observe that for all $1 \leq k \leq n$,
$$
\int_{(\f- V_{\theta} < -t )} \langle T^k\wedge \theta_{\min}^{n-k} \rangle
\leq \frac{1}{|\chi(-t)|} \int_X (-\chi) \circ (\f-V_{\theta}) \langle T^k\wedge \theta_{\min}^{n-k} \rangle 
\leq \frac{1}{|\chi(-t)|} E_{\chi}(\f).
$$
This yields the first assertion.\\

The second statement follows from similar arguments as in the K{\"a}hler case, working with the $\theta$-psh function $u:= \frac{1}{t}\f_t+\left(1-\frac{1}{t}\right) V_{\theta}$ where $\f_t:=\max(\f, V_{\theta}-t)$ for any $\f \in PSH(X,\theta)$. Let us stress that this is the only place where the assumption on the weight, $\chi\in{\mathcal W}^- \cup {\mathcal W}_M^+$ is used.
\end{proof}

\subsubsection{Alexander capacity}

For $K$ a Borel subset of $X$, we set 
$$V_{K,\theta}:= \sup\{\f\; | \; \f\in PSH(X,\theta), \; \f\leq 0 \;\; \textit{on}\;K\}.$$
Note that $$ V_{\theta}=V_{X,\theta}\leq V_{K,\theta}$$
by definition. It follows from standard arguments (see \cite[Theorem 4.2]{gz1}) that the usc regularization $V_{K,\theta}^*$ of $V_{K,\theta}$ is either a $\theta$-psh function with minimal singularities (when $K$ is not-pluripolar) or identically $+\infty$ (when $K$ is pluripolar).
\begin{defi}[Alexander-Taylor capacity]
\it{Let $K$ be a Borel subset of $X$. We set
$$T_\theta(K):= \exp(-\sup_X V_{K,\theta}^* ).$$}
\end{defi}
As in the K{\"a}hler case, the capacities $T_\theta$ and $Cap_{\theta_{\min}}$ compares as follows:

\begin{pro}\label{alex}
There exists $A>0$ such that for all Borel subsets $K\subset X$,
$$\exp\left[-\frac{A}{Cap_{\theta_{\min}}(K)}	\right] \leq T_\theta(K)\leq e \cdot \exp \left[-\left(\frac{\vol(\a)} {Cap_{\theta_{\min}}(K)}\right)^{\frac{1}{n}}  \right] $$ 
\end{pro}

\begin{proof}
It suffices to treat the case of compact sets. The second inequality is \cite[Lemma 4.2]{begz}. We prove the first inequality. We can assume that $M:=M_{\theta}(K)\geq 1$ otherwise it is sufficient to adjust the value of $A$. Let $\f$ be a $\theta$-psh function such that $\f\leq 0$ on $K$. Then $\f \leq M$ on $X$, hence
$w:=M^{-1}\left(\f-M-V_{\theta}\right)\in PSH(X,\theta_{\min})$ satisfies $\sup_X w \leq 0$ and $w \leq -1$ on $K$. We infer
$w \leq h_{K,\theta_{\min}}^*$ and
$$
w_K:=\frac{V_{K,\theta}^*-M-V_{\theta}}{M}\leq h_{K,\theta_{\min}}^* \leq 0.
$$
Then we get
\begin{eqnarray*}
Cap_{\theta_{\min}}(K)&=&   \int_X \left(- h_{K,\theta_{\min}}^*\right) \MA( V_{\theta}+h_{K,\theta_{\min}}^*) \\
&\leq &\frac{1}{M} \int_X -(V_{K,\theta}^*-M-V_{\theta})\, \MA(V_{\theta}+h_{K,\theta_{\min}}^*)\\
&\leq & \frac{C_1}{M}
\end{eqnarray*}
with $C_1>0$. The last estimate follows from Lemma below together with \cite[Proposition 1.7]{gz1} since $\sup_X(V_{K,\theta}^*-M-V_{\theta})=0$ and by \cite[Corollary 2.5]{bd}, $\langle(\theta+dd^c V_{\theta})^n\rangle={\bf 1}_{\{V_{\theta}=0\}}\theta^n\leq  C\omega^n$.
\end{proof}

The following Lemma is a straightforward generalization of \cite[Corollary 2.3]{gz1}, (see also \cite[Lemma 3.2]{bbgz}).

\begin{lem}
Let $\psi,\varphi$ be $\theta$-psh functions with minimal singularities with $\varphi$ normalized in such a way that $0\leq \varphi-V_{\theta}\leq 1$. Then we have
$$\int_X -(\psi-V_{\theta})\langle (\theta+dd^c \varphi)^{n} \rangle\leq \int_X -(\psi-V_{\theta})\langle (\theta+dd^c V_{\theta})^{n}\rangle+n\vol(\alpha).$$
\end{lem}

\subsection{Comparing Capacities}
We introduce a slighty different notion of big capacity that is comparable with respect to the usual one. For any Borel set $K\subset X$ we define
$$Cap_{\theta_{\min}}^{\lambda} (K):=\sup \left\{ \int_{K}\langle (\theta_{\min}+dd^c\psi)^n \rangle, \,\psi\in PSH(X,\theta_{\min}) \,|\, -\lambda\leq \psi \leq 0 \right\},$$
where $\lambda\geq 1$. We let the reader check that
\begin{equation}\label{nca}
Cap_{\theta_{\min}}(K)\leq Cap_{\theta_{\min}}^{\lambda}(K)\leq \lambda^n Cap_{\theta_{\min}}(K).
\end{equation}
We now compare the Monge-Amp{\`e}re capacities w.r.t. different big classes (Theo\-rem D of the introduction). 
\begin{thm}\label{compcap}
Let $\a_1$ 	and $\a_2$ be big classes on $X$ such that $\a_1\leq \a_2$.  
We assume that $\{\a_1, \a_2-\a_1\}$ satisfies Condition $\mathcal{MS}$ and that there exists a positive $(1,1)$-current $T_0\in\a_2-\a_1$ with bounded potentials. Then there exist $C>0$ such that for any Borel set $K\subset X$, 
$$\frac{1}{C}\,Cap_{\theta_{1,\min}}(K)\leq  \,Cap_{\theta_{2,\min}}(K)\leq C \left(Cap_{\theta_{1,\min}}(K)\right)^{\frac{1}{n}}.$$
\end{thm}
\noindent Note that in case of K{\"a}hler forms the result is stronger and the proof much simpler 
(see \cite[Proposition 2.5]{begz}) but we can not expect better in the general case of big classes. In the following Example \ref{exc} shows that the exponent at the right-hand side is necessary. 

\begin{proof}
Fix $\theta_1\in\a_1$, $\theta_2\in\a_2$ smooth forms. Write $T_0=(\theta_2-\theta_1)+dd^c f_0$ where $f_0$ is a bounded potential. Let $\f$ be a $\theta_1$-psh function such that $-1\leq\f-V_{\theta_1}\leq 0$ then $\f+f_0$ is a $\theta_2$-psh function. Condition $\mathcal{MS}$ insures that the potential $V_{\theta_1}+f_0$ has minimal singularities, thus there exists a positive constant $C$ such that $|V_{\theta_2}-V_{\theta_1}-f_0|\leq C$. Therefore $-\lambda\leq\f+f_0-C-V_{\theta_2}\leq 0$ where $\lambda=1+2C$. 
Now, using (\ref{nca}) and the fact that $T_1\leq T_2$ implies $\langle T_1^n\rangle \leq \langle T_2^n\rangle$  we get
\begin{eqnarray*}
\int_K \langle (\theta_1+dd^c\f)^n\rangle &\leq & \int_K \langle (\theta_2+dd^c(\f+f_0)^n\rangle
\end{eqnarray*}
namely $Cap_{\theta_{1,\min}}(K)\leq Cap_{\theta_{2,\min}}^{\lambda}(K)\leq \lambda^n Cap_{\theta_{2,\min}}(K)$ hence the left inequality. In order to prove the other inequality  we have to go through the Alexander capacity. Since $V_{\theta_1,K}^*+f_0\leq V_{\theta_2,K}^*$ 

$$\sup_X (V_{\theta_2,K}^*)\geq \sup_X (V_{\theta_1,K}^*) +\inf_X f_0,$$
and so
\begin{equation*}
T_{\theta_2}(K)\leq T_{\theta_1}(K)\cdot e^{-\inf_X f_0}.
\end{equation*}
Furthermore, using Proposition \ref{alex} we get
\begin{eqnarray*}
\exp\left[-\frac{A}{Cap_{\theta_{2,\min}}(K)}	\right] &\leq & T_{\theta_{2}}(K) \\
  &\leq &T_{ \theta_1}(K) \cdot e^{-\inf_X f_0+1}\\
  &\leq & e^{-\inf_X f_0+1} \cdot \exp \left[-\left(\frac{\vol(\a_1)} {Cap_{\theta_{1,\min}}(K)}\right)^{\frac{1}{n}}  \right]
\end{eqnarray*}
with $A$ a positive constant. Thus, there exists a constant $C>0$ such that 
\begin{eqnarray*}
Cap_{\theta_{2,\min}}(K)&\leq & A\left[\left(\frac{\vol(\a_1)}  {Cap_{\theta_{1,\min}}(K)}\right)^{\frac{1}{n}}+\inf_X f_0-1 \right]^{-1}\\
&\leq & C\, Cap_{\theta_{1,\min}}(K)^{\frac{1}{n}}.
\end{eqnarray*}
Hence the conclusion.
\end{proof}

\begin{ex}\label{exc}
Let $\pi : X \rightarrow \mathbb{P}^2$ the blow-up at one point $p$ and set $E:=\pi^{-1}(p)$. Consider $\a_1=\{\pistar \omega_{FS}\}$ and $\a_2=\{\tilde{\omega}\}$ where $\tilde{\omega}$ is a K{\"a}hler form on $X$. Let $\Delta_r$ be the polydisc of radius $r<1$ on $\mathbb{P}^2$. By \cite[Proposition 2.10]{gz1} and \cite[Lemma 4.5.8]{kli} we know that $$Cap_{\pistar \omega_{FS}}(\pi^{-1}(\Delta_r))=Cap_{\omega_{FS}}(\Delta_r)\sim \frac{1}{(-\log r)^2}.$$ Fix now a local chart $U\subset X$ such that $p\in U$ and consider $K_r\subset U$, $K_r:=\{(s,t)\in U \;|\;\, 0<\|s\|<r, \;0<\|t\|<1\}$. Then $$Cap_{\tilde{\omega}}(\pi^{-1}(\Delta_r))\geq Cap_{\tilde{\omega}}(K_r)\sim C \frac{1}{-\log r},$$ 
with $C$ a positive constant.
\end{ex}

\subsection{Energy classes with homogeneous weights}\label{hw}

As Example \ref{nostab} shows we can not hope to get stability of weighted energy classes $\mathcal{E}_\chi$ by only adding Condition $\mathcal{MS}$. We nevertheless establish a partial stability property with a gap for energy classes with respect to homogeneous weights $\chi(t)=-(-t)^p$. We recall that the functions $\chi(t)=-(-t)^p$ belong to $\mathcal{W}^-$ if $0<p\leq 1$ while they belong to $\mathcal{W}^+_M$  when $p\geq 1$.

\begin{pro}\label{quasih}
Let $\a,\beta$ be big classes. Assume that $S\in \beta$ has bounded potential and the couple $(\a,\beta)$ satisfies Condition $\mathcal{MS}$. If $p> n^2-1$ then
$$ T\in \mathcal{E}^p(X,\a) \Longrightarrow T+S\in \mathcal{E}^q(X,\a+\beta),$$
where $0<q<p-n^2+1$.
\end{pro}
\begin{proof}
Fix $\theta_{\a},\theta_{\b}$ smooth representatives of $\a,\beta$, respectively and set $\tilde{\theta}:=\theta_{\a}+\theta_{\b}$. Write $S=\theta_{\b}+dd^c\psi$ and denote $\theta_{\a,\min}:=\theta_{\a}+dd^cV_{\theta_{\a}}$ and $\tilde{\theta}_{\min}:=\tilde{\theta}+dd^cV_{\tilde{\theta}}$. We want to show that there exists a positve number $q<p$ such that given a $\theta_{\a}$-psh function $\f\in \mathcal{E}^p(X,\theta_{\a})$ then $\f+\psi\in\mathcal{E}^q(X,\tilde{\theta})$. By the first claim of Lemma $\ref{cap}$, for any $t>1$ there exists a constant $C_{\f}>0$ such that 
\begin{equation}\label{ele1}
Cap_{\theta_{\a,\min}}(\f-V_{\theta_{\a}}<-t)\leq C_{\f} t^{-(p+1)}. 
\end{equation}
The goal is to find a similar estimate from above of the quantity $Cap_{\tilde{\theta}_{\min}}(\f+\psi-V_{\tilde{\theta}}<-t)$. Set $K:= \{ \f-V_{\theta_{\a}}<-t\}$ and $\tilde{K}:=\{\f+\psi-V_{\tilde{\theta}}<-t\}$. We infer that Condition $\mathcal{MS}$ implies $\tilde{K}\subseteq K$. Thus $Cap_{\tilde{\theta}_{\min}}(\tilde{K})\leq Cap_{\tilde{\theta}_{\min}}(K)$. Now, by Theorem \ref{compcap} we know that there exists $A>0$ such that 
$$
Cap_{\tilde{\theta}_{\min}}(\tilde{K}) \leq  A \; Cap_{\theta_{\a,\min}}(K)^{\frac{1}{n}}\leq  \tilde{C}_{\f}\; t^{-\frac{p+1}{n}} 
$$
where the last inequality follows from ($\ref{ele1}$). This means that there exist $C_{\f},\varepsilon>0$ such that
$$Cap_{\tilde{\theta}_{\min}}(\tilde{K})\leq C_{\f} t^{-(n+\varepsilon+q)}$$ 
with $0<q<p-n^2+1-n\varepsilon$. Hence by Lemma \ref{cap} we get $\f+\psi\in \mathcal{E}^q(X,\tilde{\theta})$.
\end{proof}

\end{document}